\documentclass{article}

\usepackage{amsmath,amssymb}
\usepackage{amsthm}
\usepackage[pdftex,bookmarks=true]{hyperref}
\hypersetup{pdfborder={0 0 0}}
\usepackage[noadjust]{cite}
\usepackage{enumitem}
\usepackage[capitalise, noabbrev, nameinlink]{cleveref}
\usepackage{thmtools} 

\usepackage[mathlines]{lineno}

\usepackage{tikz}
\tikzset{
every node/.style={draw, circle, inner sep=2pt}
}
\usetikzlibrary{arrows}
\usetikzlibrary{decorations.pathmorphing}
\usetikzlibrary{backgrounds}

\usepackage{soul}
\usepackage{cancel}

\newtheorem{theorem}{Theorem}[section]
\newtheorem{lemma}[theorem]{Lemma}
\newtheorem{proposition}[theorem]{Proposition}
\newtheorem{corollary}[theorem]{Corollary}

\theoremstyle{definition}
\newtheorem{definition}[theorem]{Definition}
\newtheorem{observation}[theorem]{Observation}

\newtheorem{example}[theorem]{Example}

\newtheorem{question}[theorem]{Question}

\newcommand{\trans}{^\top}

\newcommand{\dunion}{\mathbin{\dot\cup}}
\newcommand{\bzero}{\mathbf{0}}

\newcommand{\bb}{\mathbf{b}}

\newcommand{\be}{\mathbf{e}}

\newcommand{\bx}{\mathbf{x}}
\newcommand{\by}{\mathbf{y}}

\newcommand{\bu}{\mathbf{u}}
\newcommand{\bv}{\mathbf{v}}
\newcommand{\bw}{\mathbf{w}}

\newcommand{\nul}{\operatorname{null}}

\newcommand{\range}{\operatorname{range}}
\newcommand{\Col}{\operatorname{Col}}

\newcommand{\vspan}{\operatorname{span}}

\newcommand{\mptn}{\mathcal{S}} 
\newcommand{\mptncl}{\mathcal{S}^{\rm cl}}

\newcommand{\GL}[2][n]{\operatorname{GL}^{(#2)}_{#1}(\mathbb{R})}

\newcommand{\excise}{\xi\xi}
\newcommand{\ttf}{\mathcal{T}_3}
\newcommand{\ttfr}{\mathcal{T}_{3,r}}

\newcommand{\mat}[1][n]{\operatorname{Mat}_{#1}(\mathbb{R})}
\newcommand{\msym}[1][n]{\operatorname{Sym}_{#1}(\mathbb{R})}

\title{A strong nullity parameter for rooted graphs}
\author{
Aida Abiad
\thanks{Department of Mathematics and Computer Science, Eindhoven University of Technology, Eindhoven, The Netherlands (a.abiad.monge@tue.nl)}
\thanks{Department of Mathematics and Data Science, Vrije Universiteit Brussel, Belgium}
\and 
Mary Flagg
\thanks{Department of Mathematics and Computer Science, University of St.~Thomas, Houston, TX, U.S.A. (flaggm@stthom.edu)}
\and
H. Tracy Hall
\thanks{Hall Labs LLC, Provo Utah, U.S.A. (h.tracy@gmail.com)}
\and
Jephian C.-H.~Lin
\thanks{Department of Applied Mathematics, National Yang Ming Chiao Tung University, Hsinchu 300093, Taiwan (jephianlin@gmail.com)}
\and 
Bryan Shader
\thanks{University of Wyoming, Colorado, U.S.A. (bshader@uwyo.edu)}
}
\date{}

\begin{document}

\maketitle

\begin{abstract}
The inverse eigenvalue problem of a graph $G$ studies the possible spectra of matrices associated with $G$, including as an important subproblem the possible nullities of such a matrix.
Much research in this area to date has focused only on the spectrum of the matrix itself, but there are applications of inverse eigenvalue problems that also involve the interaction between that spectrum and the spectrum of some maximal proper principal submatrix, or in other words the interlacing spectrum that results from crossing out any one row and the same column.
Motivated by this refined information, given a graph $G$ on $n$ vertices with a designated root vertex,
we investigate all possible nullity pairs where the first nullity is that of an $n \times n$ symmetric matrix associated to $G$ and the second nullity is that of the principal submatrix of size $(n - 1) \times (n - 1)$ that results from deleting the row and column associated to the root vertex.
We define a new parameter $\excise(G,i)$ for rooted graphs $(G,i)$ equipped with the strong nullity interlacing property that coordinates the two values of a nullity pair, and we show that this graph parameter is minor monotone.
Moreover, we prove a bifurcation lemma for the strong nullity interlacing property.
We use these new tools to characterize the rooted graphs with $\excise(G,i) \geq s$ for $s \in \{ 0,1,2,3,4, 5\}$ by finding the minimal minors for each of these families. These families turn out to have strong connections to the minimal minors for $\xi(G) \geq k$.
\end{abstract}  


\medskip


\section{Introduction}
\label{sec:intro}
Let $G$ be a graph on $n$ vertices.  The set of matrices associated to $G$ is denoted by $\mptn(G)$ and consists of all $n\times n$ real symmetric matrices whose off-diagonal $(i,j)$-entry is nonzero if and only if $\{i,j\}$ is an edge of $G$.  Here the diagonal entries can be arbitrary real numbers.
The \emph{inverse eigenvalue problem of a graph $G$} (IEP-$G$) aims to characterize all possible spectra occurs among matrices in $\mptn(G)$.  The IEP-$G$ was motivated by vibration theory \cite{Hochstadt74, GW76, Hald76, Ferguson80, BG87, MS13, GladwellIPiV05} and has strong connections to graph structure through for example the Colin de Verdi\`ere type parameters \cite{CdV, CdVF, BFH05, IEPG2}.
For more details on the IEP-$G$, see the monograph \cite{IEPGZF22}.

While the spectrum of a matrix $A\in\mptn(G)$ has been the main focus of the IEP-$G$, it is natural to consider simultaneously the spectrum of a maximal proper principal submatrix.  Let $A(i)$ be the submatrix of $A$ obtained by removing the $i$-th row and column.  The $(\lambda,\mu)$-problem, which studies the eigenvalues of $A$ (denoted by $\lambda_i$) and the eigenvalues of $A(i)$ (denoted by $\mu_i$), has received a fair amount of attention in the literature, see e.g. \cite{Hochstadt74, GW76, Hald76, Ferguson80, BG87, MS13, GladwellIPiV05, lambmu2013, BNSY14}.  Recently, the \emph{$i$-nullity pair} was introduced as $(\nul(A), \nul(A(i))$ and shown to possess a strong relationship to the rooted graph structure \cite{SNIP}.  

We say that a matrix $A\in\mptn(G)$ has \emph{the Strong Arnold Property} (or has SAP) if $X = O$ is the only symmetric matrix that satisfies $A\circ X = I\circ X = O$ and $AX = O$.  Using SAP, Colin de Verdi\`ere defined a graph parameter $\mu(G)$ and showed that it is minor monotone \cite{CdV, CdVF}; that is, $\mu(G) \leq \mu(H)$ if $G$ is a minor of $H$.  Here $G$ is a \emph{minor} of $H$ means $G$ can be obtained from $H$ by a sequence of operations of removing an edge, removing an isolated vertex, and contracting an edge.  Motivated by the Colin de Verdi\`ere parameter $\mu(G)$, the parameter $\xi(G)$ has been defined as the maximum nullity among matrices $A\in\mptn(G)$ with SAP, and this has been shown to be minor monotone as well \cite{BFH05}.  In the same paper, it was shown that $\xi(G) \geq 2$ if and only if $G$ contains $K_3$ or $K_{1,3}$ as a minor, and $\xi(G) \geq 3$ if and only if $G$ contains a minor in the $T_3$-family, as shown in \cref{fig:t3}.

In general, if a graph parameter $\zeta(G)$ is minor monotone, then a graph $G$ is said to achieve $\zeta(G) \ge k$ as a \emph{minimal minor} if $\zeta(G) \geq k$ and any proper minor $F$ of $G$ has $\zeta(F) < k$.  Thus, one may characterize graphs $G$ with $\zeta(G) \geq k$ by finding all minimal minors, of which there are only finitely many for each $k$ by the Graph Minors Theorem; see, e.g., \cite{DiestelGT}. A similar idea also works for rooted graphs.  A \emph{rooted graph} $(G,i)$ is a simple graph $G$ with a vertex $i$ designated as the root.  We say $(G,i)$ is a \emph{rooted minor} of $(H,j)$ if $(G,i)$ is obtained from a rooted graph $(H,j)$ by removing an edge, removing a vertex that is not the root, or contracting an edge, where the new vertex given by the contraction becomes the root of the minor if the contracted edge is incident to the original root.  A rooted graph parameter $\zeta(G,i)$ is then said to be \emph{minor monotone} if $\zeta(G,i) \leq \zeta(H,j)$ whenever $(G,i)$ is a rooted minor of $(H,j)$.  By an analogous graph minors theorem for rooted graphs \cite{RS10}, the set of minimal minors for $\zeta(G, i) \geq s$ is also finite for any given value of $s$.

In \cite{SNIP}, analogous statements for matrices and its submatrices were introduced and investigated.   Let $(G,i)$ be a rooted graph and let $A$ be a matrix in $\mptn(G)$.  We say $A$ has \emph{the $i$-strong nullity interlacing property} (or has $i$-SNIP) if $X = O$ is the only symmetric matrix that satisfies the identities $A\circ X = I\circ X = O$ and $(AX)(i,:] = O$.  Here $(AX)(i,:]$ designates the (non-square) submatrix of $A$ obtained by removing the $i$-th row.  We say $(G,i)$ \emph{allows} the nullity pair $(k,\ell)$ (with $i$-SNIP, respectively) if there is a matrix $A\in\mptn(G)$ that achieves the $i$-nullity pair $(k,\ell)$ (with $i$-SNIP, respectively).
It is a result of \cite{SNIP} that this new strong property enforces minor monotonicity:
If $(G,i)$ allows the nullity pair $(k,\ell)$ with $i$-SNIP and is a rooted minor of $(H,i)$, then $(H,i)$ also allows $(k,\ell)$.
Using this minor monotonicity theorem, the rooted graphs that allow $(0,0)$, $(0,1)$, $(1,1)$, or $(1,2)$ with $i$-SNIP were characterized in \cite{SNIP}, in each case in terms of minimal minors.  

In this paper we introduce a new parameter $\excise(G,i)$, which is defined as the maximum $k + \ell$ for some $k\leq \ell$ such that there is a matrix $A\in\mptn(G)$ with the $i$-nullity pair $(k,\ell)$, $k\leq \ell$, and $i$-SNIP.
Minor monotonicity can be rephrased as $\excise(G,i) \leq \excise(H,j)$ whenever $(G,i)$ is a rooted minor of $(H,j)$.
Besides minor monotonicity, we also prove a bifurcation lemma for $\excise(G,i)$,
showing that every non-negative sum up to the maximum can be achieved:
If $k + \ell \leq \excise(G,i)$ with $k\leq \ell$, then $(G,i)$ also allows $i$-nullity $(k,\ell)$ with $i$-SNIP.  As a result, this new parameter coordinates both values in a nullity pair at once---for example, $(G,i)$ allows $(1,2)$ with $i$-SNIP if and only if $\excise(G,i) \geq 3$.  
In other words, the minimal minor results of \cite{SNIP} can be restated as enumerating all of the minimal minors that characterize $\excise(G, i) \geq s$ for $s \in \{ 0,1,2,3\}$.
In the paper at hand, we advance these results by extending the characterization of $\excise(G, i) \geq s$ to $s=4$ and $s=5$.
The minimal minors for $\excise(G, i) \geq s$ for $s \in \{0, \ldots, 5\}$ are shown in \cref{fig:minminor}.
A further result of the present work is that for any $k \geq 0$, the minimal minors for $\excise(G, i) \geq 2k+1$ are obtained from the minimal minors for $\excise(G, i) \geq 2k$ by appending a leaf as the new root to the original root.
This behavior was previously observed in the minimal minors for $\excise(G, i) \geq s$ in the cases $s \in \{0, 1\}$ and $s \in \{ 2, 3\} $, but here we show that in fact this holds for any pair $s \in \{2k, 2k+1\}$. 

\begin{figure}[t]
\centering
\begin{tikzpicture}[scale=0.8, transform shape]

\node[rectangle, rounded corners=3pt, align=center, inner sep=5pt] (tau0) at (0,0) {
    $\excise(G, i) \geq 0 = 0 + 0$ \\[5pt]
    \tikz{
        \node[draw=none,rectangle,left] at (-0.5,0) {$(K_1,v)$};
        \node[fill] at (0,0) {};
    }
};

\node[rectangle, rounded corners=3pt, align=center, inner sep=5pt] (tau1) at (0,2) {
    $\excise(G, i) \geq 1 = 0 + 1$ \\[5pt]
    \tikz{
        \node[draw=none,rectangle,left] at (-0.5,0) {$(K_2,v)$};
        \node[fill] (1) at (0,0) {};
        \node (2) at (1,0) {};
        \draw (1) -- (2);
    }
};

\node[rectangle, rounded corners=3pt, align=center, inner sep=5pt] (tau2) at (5,2) {
    $\excise(G, i) \geq 2 = 1 + 1$ \\[5pt]
    \tikz{
        \node[draw=none,rectangle,left] at (-0.5,0) {$(K_3,v)$};
        \node[fill] (1) at (0,0) {};
        \node (2) at (1,0.578) {};
        \node (3) at (1,-0.578) {};
        \draw (1) -- (2) -- (3) -- (1);
    } \\[5pt]
    \tikz{
        \node[draw=none,rectangle,left] at (-0.5,0) {$(K_{1,3},\text{leaf})$};
        \node[fill] (1) at (0,0) {};
        \node (2) at (1,0.578) {};
        \node (3) at (1,-0.578) {};
        \node (4) at (0.667,0) {};
        \draw (1) -- (4) -- (2);
        \draw (4) -- (3);
    }
};

\node[rectangle, rounded corners=3pt, align=center, inner sep=5pt] (tau3) at (5,7) {
    $\excise(G, i) \geq 3 = 1 + 2$ \\[5pt]
    \tikz{
        \node[draw=none,rectangle,left] at (-0.5,0) {$(\mathsf{Paw},v)$};
        \node[fill] (1) at (0,0) {};
        \node (2) at (1,0) {};
        \node (3) at (2,0.578) {};
        \node (4) at (2,-0.578) {};
        \draw (1) -- (2) -- (3) -- (4) -- (2);
    } \\[5pt]
    \tikz{
        \node[draw=none,rectangle,left] at (-0.5,0) {$(S_{2,1,1},\text{leaf}_2)$};
        \node[fill] (1) at (0,0) {};
        \node (2) at (0.667,0) {};
        \node (3) at (1.667,0.578) {};
        \node (4) at (1.667,-0.578) {};
        \node (5) at (1.334,0) {};
        \draw (1) -- (2) -- (5) -- (3);
        \draw (5) -- (4);
    }
};

\node[rectangle, rounded corners=3pt, align=center, inner sep=5pt] (tau4) at (10,7) {
    $\excise(G, i) \geq 4 = 2 + 2$ \\[5pt]
    $T_3$-family with\\
    root at any\\
    non-cut-vertex
};

\node[rectangle, rounded corners=3pt, align=center, inner sep=5pt] (tau5) at (10,10) {
    $\excise(G, i) \geq 5 = 2 + 3$ \\[5pt]
    extending the root\\
    from $\excise(G, i) \geq 4$
};

\draw (tau0) -- (tau1) -- (tau2) -- (tau3) -- (tau4) -- (tau5);

\end{tikzpicture}
\caption{Minimal minors for $\excise$.}
\label{fig:minminor}
\end{figure}
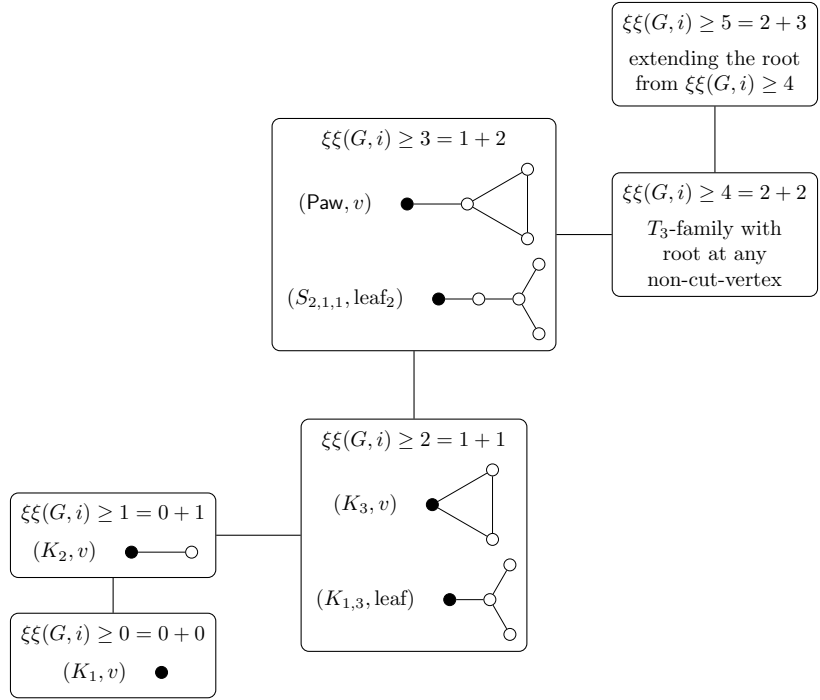

This paper is organized as follows. \cref{ssec:prelim} introduces the needed background and notation. \cref{sec:bif} presents the Bifurcation Lemma and related results, which show that the allowed nullity pair sums are consecutive. In \cref{sec:append} we show that the minimal minors of $\excise(G, i) \geq 2k + 1$ are obtained from the minimal minors of $\excise(G, i) \geq 2k$ by appending a leaf. In \cref{sec:schur}, we study how nullity pairs interact with a Schur complement construction, which is essential for the section that follows.  Building upon these results, we characterize the minimal minors of $\excise(G, i) \geq 4,5$ in \cref{sec:excise45}. Finally, we provide a comprehensive study on $i$-SNIP in \cref{sec:snipthy}, including the proof of the Bifurcation Lemma, the edge bound, and an equivalent null space definition for $i$-SNIP.
The final \cref{sec:conc} contains some concluding remarks and an open question based on observations arising in this work.

\section{Preliminaries}
\label{ssec:prelim}
Let $A$ be a real symmetric matrix and $i$ an index.  It is known that in every instance, exactly one of the following cases holds (see, e.g., \cite[Remark~2.1]{SNIP}):
\begin{itemize}
\item $\nul(A) + 1 = \nul(A(i))$, 
\item $\nul(A) = \nul(A(i))$,
\item $\nul(A) - 1 = \nul(A(i))$.
\end{itemize}
Depending on which of these three cases holds, we say the index $i$ is \emph{upper}, \emph{neutral}, or \emph{downer}, respectively.
A consequence of this characterization is that any nullity pair $(\nul(A),\nul(A(i)) = (k,\ell)$ necessarily satisfies $|k - \ell| \leq 1$.  

On the other hand, we know that the cases of $(k,k)$ and $(k+1,k)$ are equivalent in the following sense.  

\begin{proposition}[Corollaries~2.3 and 3.10 of \cite{SNIP}]
\label{prop:kk1}
Let $(G,i)$ be a rooted graph.  Then $(G,i)$ allows the nullity pair  $(k,k)$ (respectively, with $i$-SNIP) if and only if $(G,i)$ allows the nullity pair $(k+1, k)$ (respectively, with $i$-SNIP).
\end{proposition}

Therefore, we focus on the nullity pairs $(k,\ell)$ when $k \leq \ell$ and define the parameter $\excise(G,i)$ as the largest $k + \ell$ such that $(G,i)$ allows the nullity pair $(k,\ell)$ with $k \leq \ell$.  

Also, \cite[Lemma~3.6]{SNIP} states that if $A$ has $i$-SNIP, then both $A$ and $A(i)$ have SAP.  Therefore, if $(G,i)$ allows the nullity pair $(k,\ell)$ with $i$-SNIP and with $k\leq \ell$, then $k \leq \xi(G)$ and $\ell \leq \xi(G - i) \leq \xi(G)$.
This gives $\excise(G,i) \leq 2\xi(G)$, justifying the choice of notation.

In fact, $i$-SNIP can be fully characterized by checking SAP for appropriate matrices.
Denote by $E_{i,i}$ the $n \times n$ matrix whose $(i,i)$-entry is $1$ while all other entries are $0$.

\begin{theorem}[Theorem~3.9 of \cite{SNIP}]
\label{thm:snipchar}
Let $A\in\msym$ and let $i$ be an index in $\{1,\dots,n\}$.  Then the following characterization holds.
\begin{enumerate}[label={{\rm(\alph*)}}]
\item If  $i$ is a downer index of $A$, then $A$ has $i$-SNIP if and only if $A$ has SAP.
\item If $i$ is a neutral index of $A$ and $t$ is the unique value such that $i$ is a downer index of $A + tE_{i,i}$, then  $A$ has $i$-SNIP if and only if $A + tE_{i,i}$ has SAP.
 \item If $i$ is an upper index of $A$, then $A$ has $i$-SNIP if and only if $A(i)$ has SAP.
\end{enumerate}
\end{theorem}

When $A$ and $B$ are real symmetric matrices and $i$ is an index of $A$, \cite[Proposition~3.5]{SNIP} stated that the direct sum $A\oplus B$ has $i$-SNIP if and only if $A$ has $i$-SNIP and $B$ is invertible, which leads to \cref{prop:disjoint}.  This allows us to reserve attention to connected graphs.

\begin{proposition}
\label{prop:disjoint}
Let $G\dunion H$ be the disjoint union of graphs $G$ and $H$ with $i\in V(G)$.  Then 
\[
    \excise(G\dunion H, i) = \excise(G, i).
\]
\end{proposition}

Finally, we introduce notation for submatrices.  Let $A$ be a matrix, $\alpha$ a subset of row indices, and $\beta$ a subset of column indices.  Then $A[\alpha,\beta]$ stands for the submatrix of $A$ induced on the rows in $\alpha$ and columns in $\beta$.  In contrast, $A(\alpha,\beta)$ is the submatrix of $A$ obtained by removing the rows in $\alpha$ and the columns in $\beta$.  The meanings of $A(\alpha,\beta]$ and $A[\alpha,\beta)$ are also straightforward.  For convenience, we write $A[\alpha] = A[\alpha,\alpha]$, $A(\alpha) = A(\alpha,\alpha)$, and $A(i) = A(\{i\})$.    The notation~$:$ is used to denote the set of all indices, so $A(i,:]$ is the non-square submatrix of $A$ obtained by removing only row $i$ of $A$.

\section{Bifurcation Lemma}
\label{sec:bif}
Since $|\nul(A) - \nul(A(i))\,| \leq 1$ for any matrix $A$ and index $i$, all nullity pairs $(k,\ell)$ that occur are grid points in the first quadrant of the Cartesian plane lying on one of the lines $\ell = k + 1$, $\ell = k$, or $\ell = k - 1$.  

It is natural to ask the following question: If $(G,i)$ allows $(k,\ell)$, does that imply that $(G,i)$ allows $(k-1,\ell)$ or $(k,\ell-1)$?  We answer this question affirmatively (structuring the proof in three cases), and illustrate this in \cref{fig:stair}.
\begin{enumerate}
\item A rooted graph $(G,i)$ allows $(k+1,k)$ (with $i$-SNIP, respectively) if and only if it allows $(k,k)$ (with $i$-SNIP, respectively). (Observe the gray solid arrows in \cref{fig:stair}.)
\item If a rooted graph $(G,i)$ allows $(k+1,\ell+1)$, then it allows $(k,\ell)$. (Observe the gray dotted arrows in \cref{fig:stair}.)
\item If a rooted graph $(G,i)$ allows $(k+1,k+1)$ with $i$-SNIP, then it allows $(k,k+1)$ with $i$-SNIP.  Similarly, if $(G,i)$ allows $(k,k+1)$ with $i$-SNIP, then it allows $(k,k)$ with $i$-SNIP.   (Observe the black wavy arrows in \cref{fig:stair}.)
\end{enumerate}

\begin{figure}[t]
\begin{center}
\begin{tikzpicture}

\draw[-triangle 45] (-0.5,0) -- (4.5,0) node[draw=none,rectangle,right] {$\nul(A)$};
\draw[-triangle 45] (0,-0.5) -- (0,4.5) node[draw=none,rectangle,above] {$\nul(A(i))$};

\foreach \i/\j in {0/0, 1/0, 0/1, 1/1, 2/1, 1/2, 2/2, 3/2, 2/3, 3/3, 4/3, 3/4, 4/4} {
    \node[fill=white] (\i\j) at (\i,\j) {};
}

\foreach \j in {0,1,2,3} {
    \pgfmathsetmacro{\i}{int(\j + 1)}
    \draw[black!30,triangle 45-triangle 45] (\j\j) -- (\i\j);
}

\foreach \i in {2,3,4} {
    \pgfmathsetmacro{\j}{int(\i - 1)}
    \pgfmathsetmacro{\k}{int(\j - 1)}
    \draw[black!30,-triangle 45, densely dotted] (\i\j) -- (\j\k);
    \draw[black!30,-triangle 45, densely dotted] (\i\i) -- (\j\j);
    \draw[black!30,-triangle 45, densely dotted] (\j\i) -- (\k\j);
}
\draw[black!30,-triangle 45, densely dotted] (11) -- (00);

\foreach \i in {0,1,2,3} {
    \pgfmathsetmacro{\ip}{int(\i + 1)}
    \pgfmathsetmacro{\im}{int(\i - 1)}
    \draw[decoration={snake, segment length=3mm},-triangle 45, decorate] (\i\ip) -- (\i\i);
    \draw[decoration={snake, segment length=3mm},-triangle 45, decorate] (\ip\ip) -- (\i\ip);
}
\end{tikzpicture}
\end{center}
\caption{The nullity pair staircase.}
\label{fig:stair}
\end{figure}
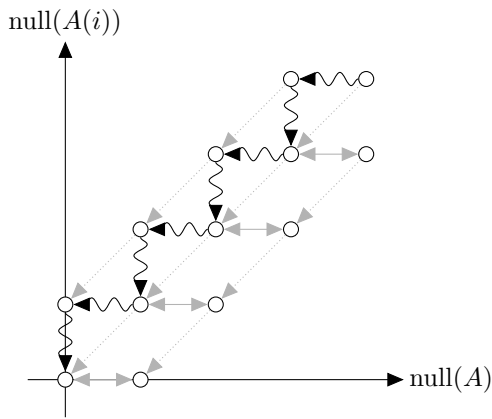

The first case follows from \cref{prop:kk1}. This is a step that does not require $i$-SNIP, but it preserves $i$-SNIP if the original matrix has $i$-SNIP.  The step in the second case does not require $i$-SNIP, and also does not necessarily preserve $i$-SNIP.
The third case both requires and preserves $i$-SNIP.
Now we prove the second case.

\begin{theorem}
\label{thm:sw}
For $k,\ell \geq 0$, if a rooted graph $(G,i)$ allows $(k+1,\ell+1)$, then it allows $(k,\ell)$.
\end{theorem}
\begin{proof}
Let $A\in\mptn(G)$ with the $i$-nullity pair $(k+1, \ell+1)$.  We first claim that there is a nonzero vector $\bv\in\ker(A)$ such that its $i$-th entry $(\bv)_i$ is $0$.  To see this, we consider two cases.  First, we consider the case when $k+1 \geq 2$.  Since any basis of $\ker(A)$ contains at least two vectors, we may obtain a linear combination of them to create a nonzero vector $\bv\in\ker(A)$ with $(\bv)_i = 0$.  Second, we consider the case when $k+1 = 1$ and $\ker(A) = \vspan\{\bv\}$ for some vector $\bv$.  If $(\bv)_i = 0$, then we are done.
Otherwise, $(\bv)_i \neq 0$ implies that column $i$ of $A$ is a linear combination of other columns.  Since $A$ is symmetric, this means $\nul(A(i)) = \nul(A) - 1 = 0 = \ell + 1$, which contradicts the assumption $\ell \geq 0$.

Let $\bv\in\ker(A)$ be a nonzero vector with $(\bv)_i = 0$.  Let $j$ be an index such that $(\bv)_j \neq 0$.  Necessarily, $j \neq i$.  We claim that $A' = A + E_{j,j}$ has the $i$-nullity pair $(k,\ell)$.  To see this, we observe that $(\bv)_j \neq 0$ implies that column $j$ of $A$ is a linear combination of other columns, so $\nul(A(j)) = \nul(A) - 1 = k$.
On the other hand, after row and column operations on $A'$, column $j$ and row $j$ are left with a single nonzero entry $1$ occupying the $(i,i)$-position.
Therefore, $\nul(A') = \nul(A(j)) = k$.  On the other hand, let~$\bu$ be the vector obtained from~$\bv$ by removing entry $i$.
Then $\bu\in\ker(A(i))$ is a nonzero vector with $(\bu)_j \neq 0$.
For the same reason, we have
\[
\nul(A'(i)) = \nul(A(\{i,j\}) = \nul(A(i)) - 1 = \ell.
\]
\end{proof}

For the third case, we need the Bifurcation Lemma (\cref{thm:bifur}), which will be proved in \cref{sec:snipthy}.

\begin{restatable}[Bifurcation Lemma]{theorem}{bifurcation}
\label{thm:bifur}
Let $(G,i)$ be a rooted graph on $n$ vertices and $A \in \mptn(G)$ a matrix with $i$-SNIP.  Then there is an $\epsilon > 0$ such that for any $M\in\msym$ with $\|M - A\| < \epsilon$, a matrix $A'\in\mptn(G)$ exists with $i$-SNIP such that $M$ and $A'$ have the same $i$-nullity pair.
\end{restatable}

\begin{lemma}[West Lemma]
\label{lem:west}
For $k \geq 0$, if a rooted graph $(G,i)$ allows nullity pair $(k+1,k+1)$ with $i$-SNIP, then it allows nullity pair $(k,k+1)$ with $i$-SNIP. 
\end{lemma}
\begin{proof}
Let $A$ be a matrix in $\mptn(G)$ with $i$-SNIP that achieves the $i$-nullity pair $(k+1,k+1)$ for some $k \geq 0$.
For convenience, we may without loss of generality assume $i = 1$ and write  
\[
    A = \begin{bmatrix} a & \bb\trans \\
    \bb & C
    \end{bmatrix}.   
\]
Since $i$ is a neutral index of $A$, we have $\bb\in\Col(C)$.  Since $k\geq 1$, there is a nonzero vector $\bv\in\ker(C)$.  Consider the matrix 
\[
    M = \begin{bmatrix} a & \bb\trans + \epsilon\bv\trans \\
    \bb + \epsilon\bv & C 
    \end{bmatrix}.   
\]
For $\epsilon\neq 0$, we now have $\bb + \epsilon\bv\notin\Col(C)$, so $i$ is an upper index of $M$ and the $i$-nullity pair of $M$ is $(k,k+1)$. 
 By the Bifurcation Lemma (\cref{thm:bifur}), there is a matrix $A'\in\mptn(G)$ that has $i$-SNIP and that achieves the $i$-nullity pair $(k,k+1)$.  
\end{proof}

\begin{lemma}[South Lemma]
\label{lem:south}
For $k \geq 0$, if a rooted graph $(G,i)$ allows $(k,k+1)$ with $i$-SNIP, then it allows $(k,k)$ with $i$-SNIP.
\end{lemma}
\begin{proof}
Let $A$ be a matrix in $\mptn(G)$ with $i$-SNIP that achieves the $i$-nullity pair $(k,k+1)$.  For convenience, we may without loss of generality assume $i = 1$ and write  
\[
    A = \begin{bmatrix} a & \bb\trans \\
    \bb & C
    \end{bmatrix}.   
\]
Since $i$ is an upper index of $A$, we have $\bb\notin\Col(C)$.  Equivalently, there is a vector $\bv\in\ker(C)$ such that $\bb\trans\bv\neq 0$.  Consider the matrix 
\[
    M = \begin{bmatrix} a & \bb\trans \\
    \bb & C + \epsilon\bb\bb\trans
    \end{bmatrix}.   
\]
For $\epsilon\neq 0$, we now have 
\[
    (C + \epsilon\bb\bb\trans)\bv = \bzero + \epsilon(\bb\trans\bv)\bb,
\]
so $\bb\in\Col(C + \epsilon\bb\bb\trans)$ and $i$ is a downer or neutral index for $A'$.  Note that adding a small rank-$1$ perturbation can only cause the nullity to change by at most one, and that when the perturbation is small, the nullity can only stay the same or be reduced.  In summary, $\nul(M) \in \{ k, k-1\}$ and $\nul(M(i)) \in \{ k+1, k\}$.  The only possibility to make $i$ downer or neutral is $\nul(M) = k = \nul(M(i))$.  By the Bifurcation Lemma (\cref{thm:bifur}), there is a matrix $A'\in\mptn(G)$ that has $i$-SNIP and that achieves the same $i$-nullity pair as $M$, namely $(k,k)$.  
\end{proof}

\begin{example}
\label{ex:bifurstar}
Consider the star graph $K_{1,5}$.  Let $(K_{1,5}, \text{center})$ and $(K_{1,5}, \text{leaf})$ be rooted graphs based on $K_{1,5}$ with the root at the center and an arbitrary leaf, respectively.  If we only focus on the nullity pairs $(k,\ell)$ with $k\leq \ell$, then by \cite[Proposition~4.6]{SNIP} the allowed nullity pairs are as follows.
\begin{itemize}
\item $(K_{1,5}, \text{center})$ allows $\underline{(0,0)}$, $\underline{(0,1)}$, $(1,2)$, $(2,3)$, $(3,4)$, and $(4,5)$.
\item $(K_{1,5}, \text{leaf})$ allows $\underline{(0,0)}$, $\underline{(0,1)}$, $\underline{(1,1)}$, $(2,2)$, and $(3,3)$.
\end{itemize}
Here the nullity pairs that are underlined are those that allow $i$-SNIP by \cite[Theorem~5.12]{SNIP}.  We note that $(K_{1,5}, \text{center})$ allows $(1,2)$ but not $(1,1)$, indicating that $i$-SNIP is necessary in \cref{lem:south}.
On the other hand, $(K_{1,5}, \text{leaf})$ allows $(2,2)$ but not $(1,2)$, indicating that $i$-SNIP is necessary in \cref{lem:west}.  Similar behaviors can also be observed on larger star graphs.
\end{example}

\section{Extending the root}
\label{sec:append}
We say that a rooted graph $(G,i)$ is obtained from a rooted graph $(H,j)$ by \emph{extending the root} if $G$ is obtained from $H$ by appending a leaf $i$ to $j$ and reassigning the root to $i$.  
Referring to \cref{fig:minminor}, we observe that the minimal minors for $\excise(G, i) \geq 2k + 1$ are exactly those rooted graphs obtained from the minimal minors of $\excise(G, i) \geq 2k$ by extending the roots.  This observation is true for at least $k \in \{ 0, 1\}$ by \cite{SNIP}.
It turns out that the observation is true in general for all $k \geq 0$.

We first show that minimal minors for $\excise(G, i) \geq 2k + 1$ always have their roots at leaves.

\begin{lemma}
\label{lem:oddleaf}
If $(G,i)$ is a minimal minor of $\excise(G, i) \geq 2k + 1$ for some $k \geq 0$, then $i$ is a leaf of $G$.
\end{lemma}
\begin{proof}
For fixed $k \geq 0$, let $(G,i)$ achieve $\excise(G, i) \geq 2k + 1$ as a minimal minor.
Without loss of generality, we may assume that $i = 1$.  Let $A$ be a matrix in $\mptn(G)$ with $i$-SNIP that achieves the $i$-nullity pair $(k,k+1)$.  Since $k \geq 0$ and $\nul(A(i)) = k + 1 \geq 1$, $G - i$ has at least one vertex, and we may assume $G$ has $n + 1$ vertices with $n \geq 0$.  We may write  
\[
    A = \begin{bmatrix}
    a_{1,1} & \bb\trans \\
    \bb & C
    \end{bmatrix}.
\]
Since $i$ is an upper index, we have $\bb\notin\Col(C)$; equivalently, there is a vector $\bv\in\ker(A)$ such that $\bb\trans\bv \neq 0$.  Let $j$ be an index where $\bb$ and $\bv$ are nonzero and let $\be_j\in\mathbb{R}^{n}$ be the vector whose entry $j$ is equal to $1$ while all other entries are zero.  Then the matrix  
\[
    A' = \begin{bmatrix}
    a_{1,1} & \be_j\trans \\
    \be_j & C
    \end{bmatrix}
\]
has the $i$-nullity pair $(k,k+1)$ as well since $\be_j\notin\Col(C)$.  

Since $i$ is an upper index of $A$ and $A'$, $A$ has $i$-SNIP if and only if $C$ has SAP if and only if $A'$ has $i$-SNIP by \cref{thm:snipchar}.  Now let $G'$ be the graph obtained from $G$ by removing all edges incident to $i$ except for $\{i,j\}$.  Then $A'\in\mptn(G')$, and $(G',i)$ is a rooted minor of $(G,i)$ that allows the same nullity pair $(k,k+1)$ with $i$-SNIP.  This means that if $i$ is not a leaf, then $(G,i)$ is not minor minimal for $\excise \geq 2k + 1$.  Equivalently, every minimal minor for $\excise(G, i) \geq 2k + 1$ has the root at a leaf.
\end{proof}

The next lemma makes a connection between the odd case and the even case.

\begin{lemma}
\label{lem:updown}
Let $G$ be a graph on $n + 1$ vertices, let $i\in V(G)$ be a leaf of $G$ with $j$ its unique neighbor, and let $H = G - i$.
Let $A$ be a matrix in $(G)$ with $B = A(i)$ a matrix in $\mptn(H)$.  Then $i$ is an upper index of $A$ if and only if $j$ is a downer index of $B$.  Moreover, $A$ has $i$-SNIP if and only if $B$ has $j$-SNIP.
\end{lemma}
\begin{proof}
For convenience, we may assume $i = 1$, $j = 2$, and write 
\[
    A = \begin{bmatrix}
        a_{1,1} & k & \bzero\trans \\
        k & a_{2,2} & \bb\trans \\
        \bzero & \bb & C.
    \end{bmatrix}.
\]
Note that $A(i,i] = k\be_j\in\mathbb{R}^n$.  Recall that $i$ is an upper index of $A$ if and only if $k\be_j\notin\Col(B)$.  Now we study the cases when $j$ is upper, neutral, and downer for $B$.  

When $j$ is upper, we have $\bb\notin\Col(C)$.  This means that there is $\bv\in\ker(C)$ such that $\bb\trans\bv\neq 0$.  Thus, 
\[
    B\begin{bmatrix} 0 \\ \bv \end{bmatrix} = 
    \begin{bmatrix}
    a_{2,2} & \bb\trans \\
    \bb & C
    \end{bmatrix}\begin{bmatrix} 0 \\ \bv \end{bmatrix} = 
    \begin{bmatrix} \bb\trans\bv \\ \bzero \end{bmatrix}.
\]
In this case, $k\be_j\in\Col(B)$.

When $j$ is neutral, we have $\bb = C\bx$ for some $\bx$ and $a_{2,2} = \bx\trans C\bx + p$ for some $p \neq 0$.  Thus we have
\[
    B\begin{bmatrix} 1 \\ -\bx \end{bmatrix} = 
    \begin{bmatrix}
    \bx\trans C\bx + p & \bx\trans C \\
    C\bx & C
    \end{bmatrix}\begin{bmatrix} 1 \\ -\bx \end{bmatrix} = 
    \begin{bmatrix} p \\ \bzero \end{bmatrix}.
\]
In this case, $k\be_j\in\Col(B)$.

When $j$ is downer, we have $\bb = C\bx$ for some $\bx$ and $a_{2,2} = \bx\trans C\bx$.  Then $\bv = \begin{bmatrix} 1 \\ -\bx \end{bmatrix}$ is in $\ker(C)$ and $k\be_j\trans\bv = k \neq 0$.  Thus, in this case, $k\be_j\notin\Col(B)$.

As we have gone through all the possible cases, we know that $i$ is an upper index of $A$ if and only if $k\be_j\notin\Col(B)$ if and only if $j$ is a downer index of $B$.

Finally, $A$ with the upper index $i$ has $i$-SNIP if and only if $A(i) = B$ has SAP if and only if $B$ with the downer index $j$ has $j$-SNIP by \cref{thm:snipchar}.
\end{proof}

\begin{corollary}
\label{cor:updown}
Let $(G,i)$ be obtained from $(H,j)$ by extending the root.  Then $\excise(G,i) \geq 2k + 1$ if and only if $\excise(H,j) \geq 2k$.  
\end{corollary}
\begin{proof}
Suppose $\excise(G,i) \geq 2k + 1$.  Then $(G,i)$ allows $(k,k+1)$ with $i$-SNIP.  By \cref{lem:updown}, $(H,j)$ allows $(k+1,k)$ with $i$-SNIP.  By \cref{prop:kk1}, $(H,j)$ allows $(k,k)$ with $i$-SNIP and $\excise(H,j) \geq 2k$.  

Conversely, suppose $\excise(H,j) \geq 2k$.  Then $(H,j)$ allows $(k,k)$ with $i$-SNIP.  By \cref{prop:kk1}, $(H,j)$ allows $(k+1,k)$ with $i$-SNIP as well.  By \cref{lem:updown}, $(G,i)$ allows $(k,k+1)$ with $i$-SNIP.
\end{proof}

Note that $\excise(P_3, \text{leaf}) = \excise(P_2, \text{leaf}) = 1$, so $\excise(G,i) = 2k + 1$ does not necessarily imply $\excise(H,j) = 2k$.

Now we show that the minimal minors for $\excise(G, i) \geq 2k$ and $\excise(G, i) \geq 2k + 1$ can be derived from each other.

\begin{theorem}
\label{thm:2k2kp1}
For each $k \geq 0$, the minimal minors for $\excise(G, i) \geq 2k + 1$ are exactly those rooted graphs obtained from the minimal minors for $\excise(G, i) \geq 2k$ by extending their roots.
\end{theorem}
\begin{proof}
Let $(G,i)$ achieve $\excise(G, i) \geq 2k + 1$ as a minimal minor.
By \cref{lem:oddleaf}, $i$ is a leaf of $G$.  Let $j$ be the unique neighbor of $i$ in $G$ and let $H = G - i$.  By \cref{cor:updown}, $\excise(H,j) \geq 2k$.  If $(H',j')$ is a proper rooted minor of $(H,j)$, then the rooted graph $(G',i')$ obtained from $(H',j')$ by extending the root is a proper rooted minor of $(G,i)$.  If $\excise(H',j') \geq 2k$, then  $\excise(G',i') \geq 2k+1$ by \cref{cor:updown}, violating the minimality of $(G,i)$.  Thus, $(H,j)$ achieves $\excise(H, j) \geq 2k$ as a minimal minor.  

Conversely, suppose that $(H,j)$ achieves $\excise(H, j) \geq 2k$ as a minimal minor.  Let $(G,i)$ be obtained from $(H,j)$ by extending the root.  By \cref{cor:updown}, $\excise(G,i) \geq 2k + 1$.  Let $(G',i')$ be a proper minor of $(G,i)$.  If $i'$ is not a leaf, then the edge $\{i,j\}$ was either deleted or contracted in $(G,i)$ in order to obtain $(G',i')$.  In the former case, $i = i'$ is an isolated vertex and $\excise(G',i') = 0 < 2k+1$ by \cref{prop:disjoint}.  In the latter case, $(G',i')$ is a minor of $(H,j)$, which means $\excise(G',i') \leq \excise(H,j) \leq 2k$.  If $i'$ is a leaf with the unique neighbor $j'$,  then the rooted graph $(H',j')$ with $H' = G' - i'$ is a proper minor of $(H,j)$.  If $\excise(G',i') \geq 2k + 1$, then $\excise(H',j') \geq 2k$ by \cref{lem:updown}, violating the minimality of $(H,j)$.  Thus, $(G,i)$ achieves $\excise(G, i) \geq 2k + 1$ as a minimal minor.
\end{proof}

\section{Schur complement}
\label{sec:schur}
In this section, we provide some general results about the Schur complement that leads to many applications.

Let $A$ be a real symmetric matrix of the form  
\[
    A = \begin{bmatrix}
        Q & B\trans \\
        B & C
    \end{bmatrix}
\]
such that $Q$ is invertible.  Then the \emph{Schur complement} of $Q$ in $A$ is denoted and defined by 
\[
    A/Q = C - BQ^{-1}B\trans.
\]
Note that the notion of the Schur complement is not limited to square matrices in general (though $Q$ has to be square to be invertible), but we only focus on real symmetric cases.  Also, the intuition behind the Schur complement is row and column operations
\[
    \begin{bmatrix}
        I & O \\
        -BQ^{-1} & I
    \end{bmatrix}
    \begin{bmatrix}
        Q & B\trans \\
        B & C
    \end{bmatrix}
    \begin{bmatrix}
        I & -Q^{-1}B\trans \\
        O & I
    \end{bmatrix} = 
    \begin{bmatrix}
    Q & O \\
    O & A / Q
    \end{bmatrix}.
\]
Consequently, we have $\nul(A) = \nul(A/Q)$.  

\begin{proposition}
\label{prop:schur}
Let $A$ be an $n\times n$ real symmetric matrix such that $A[\alpha]$ is invertible for some $\alpha\subseteq[n]$.  Then for any $i\notin\alpha$, $A$ and $A/A[\alpha]$ have the same $i$-nullity pair.
\end{proposition}
\begin{proof}
By the property of the Schur complement, $\nul(A) = \nul(A/A[\alpha])$.  Without loss of generality, we may assume $\alpha$ are the first few indices and write  
\[
    A = \begin{bmatrix}
        A[\alpha] & B\trans \\
        B & C
    \end{bmatrix}
\]
and thus $A/A[\alpha] = C - BA[\alpha]^{-1}B\trans$.  By direct computation, 
\[
    (A/A[\alpha])(i) = C(i) - (BA[\alpha]^{-1}B\trans)(i) = A(i)/A[\alpha],
\]
so $\nul((A/A[\alpha])(i)) = \nul(A(i))$.  In summary, $A$ and $A/A[\alpha]$ have the same $i$-nullity pair.
\end{proof}

In \cite[Lemma~3.5]{BFH05} and \cite[Lemma~3.2]{HvdH07}, it was shown that $A$ having SAP implies that $A/A[\alpha]$ also has SAP, under some assumptions.  Here we make a general statement for this behavior and extend it to $i$-SNIP.  


Recall that for a subset $\alpha$ of vertices of $G$,
\[
    N(\alpha) = \bigcup_{v\in\alpha}N(v) \setminus \alpha.
\]

\begin{lemma}
\label{lem:schursnip}
Let $G$ be a graph and $A\in\mptn(G)$.  Suppose $\alpha\subseteq V(G)$ is a subset of vertices such that $A[\alpha]$ is invertible and $A/A[\alpha]\in\mptn(H)$ for some graph $H$ in which $N(\alpha)$ is a clique.  Then $A$ having SAP implies that $A/A[\alpha]$ has SAP, and for any given $i\notin\alpha$, $A$ having $i$-SNIP implies that $A/A[\alpha]$ has $i$-SNIP.  
\end{lemma}
\begin{proof}
Let $a = |\alpha|$ and $b = |N(\alpha)|$.  Without loss of generality, we may assume the first rows/columns are corresponding to $\alpha$ and write 
\[
    A = 
        \begin{bmatrix}
        Q & B\trans \\
        B & C
    \end{bmatrix}
\]
with $Q = A[\alpha]$.  We may further assume the first $b$ rows in $B$ correspond to $N(\alpha)$, so $B$ is zero except for the first $b$ rows.  By the Schur complement, we have 
\[
    \begin{bmatrix}
        I & O \\
        -BQ^{-1} & I
    \end{bmatrix}
    \begin{bmatrix}
        Q & B\trans \\
        B & C
    \end{bmatrix}
    \begin{bmatrix}
        I & -Q^{-1}B\trans \\
        O & I
    \end{bmatrix} = 
    \begin{bmatrix}
    Q & O \\
    O & A / Q
    \end{bmatrix},
\]
which we write as $E\trans A E = Q\oplus (A/Q)$.

Let $X$ be a real symmetric matrix such that $A/Q\circ X = I\circ X = O$.  Define  
\[
    \widehat{X} = E(O\oplus X)E\trans = 
    \begin{bmatrix}
        Q^{-1}B\trans XBQ^{-1} & -Q^{-1}B\trans X\\
        -XBQ^{-1} & X
    \end{bmatrix}.
\]
We claim that $A\circ\widehat{X} = I\circ\widehat{X} = O$.  
Since $A/Q\in\mptn(H)$ and $N(\alpha)$ is a clique in $H$, we have $X[N(\alpha)] = O$, which further implies $C\circ X = O$ as $H$ is a supergraph of $G - \alpha$.  
Since $B$ is zero except for the first $b$ rows, each column of $XB$ is a linear combination of the first $b$ columns of $X$.  The first $b$ columns of $X$ have their first $b$ entries zero by $X[N(\alpha)] = O$, so the first $b$ rows of $XB$ are zero.  Consequently, the first $b$ rows of $-XBQ^{-1}$ and the first $b$ columns of $-Q^{-1}B\trans X$ are zero.  Moreover, as $B\trans$ is zero except for the first $b$ columns and $XB$ is zero on the first $b$ rows, we have $B\trans XB = O$ and $Q^{-1}B\trans XBQ^{-1} = O$.  Combining all these observations, we have $A\circ\widehat{X} = I\circ\widehat{X} = O$.  

For SAP, it is straightforward to see that if $(A/Q)X = O$, then 
\[
    (E\trans AE)(O\oplus X) = (Q\oplus A/Q)(O\oplus X) = O
\]
and
\[
    A\widehat{X} = AE(O\oplus X)E\trans = (E\trans)^{-1}(E\trans AE)(O\oplus X)E\trans = O.
\]
In summary, if $A/Q\circ X = I\circ X = O$ and $(A/Q)X = O$, then $A\circ\widehat{X} = I\circ\widehat{X} = O$ and $A\widehat{X} = O$.  Suppose $A$ has SAP.  Then $\widehat{X} = O$, which implies $O\oplus X = O$ and $X = O$.  Therefore, $A/Q$ has SAP.  

For $i$-SNIP, we use the notation $\overset{(i,:]}{=} O$ to indicate that a matrix is zero except for possibly the $i$-th row.  Following a similar argument, if $(A/Q) X \overset{(i,:]}{=} O$, then $(E\trans AE)(O\oplus X) \overset{(i,:]}{=} O$, which implies $(E\trans AE)(O\oplus X)E\trans \overset{(i,:]}{=} O$.  Since each row of $(E\trans)^{-1}$ has entry $i$ equal to $0$ except possibly for the $i$-th row,
\[
    A\widehat{X} = AE(O\oplus X)E\trans = (E\trans)^{-1}(E\trans AE)(O\oplus X)E\trans \overset{(i,:]}{=} O.
\]
Therefore, if $A$ has $i$-SNIP, then $A/Q$ has $i$-SNIP.  
\end{proof}

This lemma is particularly useful when $|N(\alpha)| = 1$, which implies that $N(\alpha)$ is a clique in $H$.  

We end this section by showing some reduction formulas of $\excise(G, i)$ with respect to cut-vertices.
Let $G$ and $H$ be two graphs such that there is a vertex labeled $v$ on each of them.  Then the \emph{vertex sum} of $G$ and $H$ at $v$, denoted by $G\oplus_v H$, is the graph obtained from $G$ and $H$ by identifying the vertices labeled by $v$, all other vertices remaining distinct in either $G$ or $H$.
Consequently, $v$ is a cut-vertex of $G\oplus_v H$.

\begin{theorem}
\label{thm:cutvtx}
Let $G = G_1\oplus_i G_2$ be a graph such that $G_1$ and $G_2$ are connected graphs on at least $2$ vertices.  Then  
\[
    \excise(G,i) = \max\{\excise(G_1,i), \excise(G_2,i)\}.
\]
\end{theorem}
\begin{proof}
Since $G_1$ and $G_2$ are subgraphs of $G$, we have $\excise(G_1, i) \leq \excise(G, i)$ and $\excise(G_2, i) \leq \excise(G, i)$ by minor monotonicity.  

On the other hand, let $A\in\mptn(G)$ be a matrix whose $i$-nullity pair is $(k,\ell)$ with $k + \ell = \excise(G, i)$.  By reordering the vertices we may assume $A$ has the form  
\[
    A = \begin{bmatrix}
        C_1 & \bb_1 & O \\
        \bb_1\trans & a_{i,i} & \bb_2\trans \\
        O & \bb_2 & C_2
    \end{bmatrix}.
\]
Since $A$ has $i$-SNIP, $A(i) = C_1\oplus C_2$ has SAP.  By \cite[Lemma~3.1]{BFH05}, one of $C_1$ and $C_2$ is invertible.  If $C_1$ is invertible, then $A/C_1\in\mptn(G_2)$ has the same $i$-nullity pair as $A$ and has $i$-SNIP  by \cref{prop:schur} and \cref{lem:schursnip}, which means $\excise(G,i) \leq \excise(G_2, i)$.  Similarly, if $C_2$ is invertible, we have $\excise(G, i) \leq \excise(G_1, i)$.  The desired equality therefore holds. 
\end{proof}

\begin{corollary}
Let $G$ be a graph with a cut-vertex $i$
such that $\excise(G, i) = k$.  Then $(G,i)$ does not achieve $\excise(G, i) \ge k$ as a minimal minor.
\end{corollary}
\begin{proof}
We may assume that $G = G_1\oplus_i G_2$.  By \cref{thm:cutvtx},
at least one of
\[
\excise(G_1, i) = k \ \ \mbox{ or }\ \ \excise(G_2, i) = k
\]
must hold.
But $(G_1,i)$ or $(G_2,i)$ are both proper subgraphs of $G$ with the same root vertex. 
\end{proof}

\begin{theorem}
Let $G = G_1\oplus_{v_1}G_0\oplus_{v_2}G_2$ be a graph such that $G_1$ and $G_2$ are connected graphs on at least $2$ vertices. (We allow $v_1 = v_2$.)
Let $i\in V(G_0)$.  Then 
\[
    \excise(G,i) = \max\{
    \excise(K_2\oplus_{v_1}G_0\oplus_{v_2}G_2, i), 
    \excise(G_1\oplus_{v_1}G_0\oplus_{v_2}K_2, i)
    \}.
\]
\end{theorem}
\begin{proof}
Name the two graphs in the maximum as $H_1 = K_2\oplus_{v_1}G_0\oplus_{v_2}G_2$ and $H_2 = G_1\oplus_{v_1}G_0\oplus_{v_2}K_2$.  Since $G_1$ is connected and contains vertices other than $v_1$, by contracting these vertices, $H_1$ is a rooted minor of $G$.  Similarly, $H_2$ is a rooted minor of $G$.  This implies
\[
    \excise(G,i) \geq \max\{
    \excise(H_1,i), \excise(H_2,i)
    \},
\]
so it suffices to show that at least one of $\excise(G,i) \leq \excise(H_1,i)$ or $\excise(G,i) \leq \excise(H_2,i)$ holds.

Let $A\in\mptn(G)$ be a matrix with $i$-SNIP that achieves the $i$-nullity pair whose sum is $\excise(G,i)$.  We may partition $V(G)$ by $V(G_1 - v_1)$, $V(G_0)$, and $V(G_2 - v_2)$, and then write $A$ as 
\[
    A = 
    \begin{bmatrix}
        C_1 & B_1 & O \\
        B_1\trans & A_0 & B_2\trans \\
        O & B_2 & C_2 
    \end{bmatrix},
\]
where column $v_i$ is the only nonzero column in $B_i$ for $i \in \{1,2\}$.  Now we consider two cases based on the invertibility of $C_1$ and $C_2$.  

{\bf Case 1: One of $C_1$ and $C_2$ is invertible.} Suppose $C_1$ is invertible.  Let $\alpha = V(G_1 - v_1)$.  Then $A[\alpha] = C_1$, $A/C_1\in\mptn(G_0\oplus_{v_2}G_2)$, and $N(\alpha) = \{v_1\}$ is a clique in $G_0\oplus_{v_2}G_2$. 
 By \cref{prop:schur} and \cref{lem:schursnip}, $A/C_1$ has $i$-SNIP and achieves the same $i$-nullity pair as $A$.  Thus, $\excise(G,i) \leq \excise(G_0\oplus_{v_2}G_2, i) \leq \excise(H_1,i)$.  Similarly, if $C_2$ is invertible, then $\excise(G,i) \leq \excise(G_1\oplus_{v_1}G_0, i) \leq \excise(H_2,i)$.

{\bf Case 2: Both $C_1$ and $C_2$ are singular.} Suppose there are nonzero vectors $\bx_1\in\ker(C_1)\cap\ker(B_1\trans)$ and $\bx_2\in\ker(C_2)\cap\ker(B_2\trans)$.  Then the matrix 
\[
    X = 
    \begin{bmatrix}
        O & O & \bx_1\bx_2\trans \\
        O & O & O \\
        \bx_2\bx_1\trans & O & O
    \end{bmatrix}
\]
is a nonzero real symmetric matrix satisfying $A\circ X = I\circ X = O$ and $AX = O$.  Therefore, $A$ does not have SAP or $i$-SNIP, which violates our assumption.  Consequently, either $\ker(C_1)\cap\ker(B_1\trans) = \{\bzero\}$ or $\ker(C_2)\cap\ker(B_2\trans) = \{\bzero\}$.  Suppose $\ker(C_1)\cap\ker(B_1\trans) = \{\bzero\}$.  Since $B_1$ has a unique nonzero column and $\nul(B_1\trans) = |V(G_1 - v_1)| - 1$, we have $\nul(C_1) = 1$.
Let $\ker(C_1) = \vspan(\{\bx_1\})$ and let $\by_1$ be column $v_i$ of $B_1$.  Then $\ker(B_1\trans) = \vspan(\{\by_1\})^\perp$, and $\ker(C_1)\cap\ker(B_1\trans) = \{\bzero\}$ implies $\by_1\trans\bx_1 \neq 0$ and $\by_1\notin\vspan(\{\bx_1\})^\perp = \Col(C_1)$.  Since $\nul(C_1) = 1$ and $\by_1\notin\Col(C_1)$, we have $\nul(A[G_1]) = 0$.
Also, since $C_1$ is symmetric and $\nul(C_1) = 1$, there is a vertex $u_1\in V(G_1 - v_1)$ such that $C_1(u_1)$ is invertible.
Equivalently, $A[\alpha] = C_1[\alpha]$ is invertible, where $\alpha = V(G_1 - v_1 - u_1)$.  Thus, $A[G_1] / A[\alpha]$ is a $2\times 2$ matrix with the $v_1$-nullity pair $(\nul(A[G_1]), \nul(C_1)) = (0,1)$, which means its $(u_1,u_1)$-entry is zero and its $(u_1,v_1)$-entry is nonzero.
By direct computation, the $(u_1,u_1)$-entry and the $(u_1,v_1)$-entry of $A / A[\alpha]$ are the same as those in $A_1[G_1] / A[\alpha]$.  Consequently, $A/A[\alpha]\in\mptn(H_1)$ and $N(\alpha) \subseteq \{u_1,v_1\}$ is a clique in $H_1$.
By \cref{prop:schur} and \cref{lem:schursnip}, $A/A[\alpha]$ has $i$-SNIP and achieves the same $i$-nullity pair as $A$.  Thus, $\excise(G,i) \leq \excise(H_1,i)$.  Similarly, if $\ker(C_2)\cap\ker(B_2\trans) = \{\bzero\}$, then we have $\excise(G,i) \leq \excise(H_2,i)$.

Combining the two possible cases, the desired equality holds.
\end{proof}

\section{Minimal minors for \texorpdfstring{$\excise(G, i)\geq 4,5$}{xixi(G, i) >= 4,5}}
\label{sec:excise45}
In this section we  characterize the rooted graphs $(G,i)$ with $\excise(G,i)\geq 4$ and those with $\excise(G,i) \geq 5$.  The $T_3$-family, as shown in Figure~\ref{fig:t3}, was introduced in \cite{HvdH07} as the complete list of minimal minors for $\xi(G) \geq 3$.  That is, $\xi(G) \geq 3$ if and only if $G$ contains a minor in the $T_3$-family.  In Figure~\ref{fig:t3}, every vertex that is not a cut-vertex is filled.
We will show that the set of choices of a rooted graph $(G,i)$ where $G$ is in the $T_3$-family and $i$ is not a cut-vertex forms the complete list of minimal minors for $\excise(G, i) \geq 4$.
Equivalently, a rooted graph $(G,i)$ has $\excise(G,i) \geq 4$ if and only if $(G,i)$ contains a rooted minor in the $T_3$-family rooted at a non-cut-vertex.
By \cref{thm:2k2kp1}, this also leads to the complete list of minimal minors for $\excise(G, i) \geq 5$.  

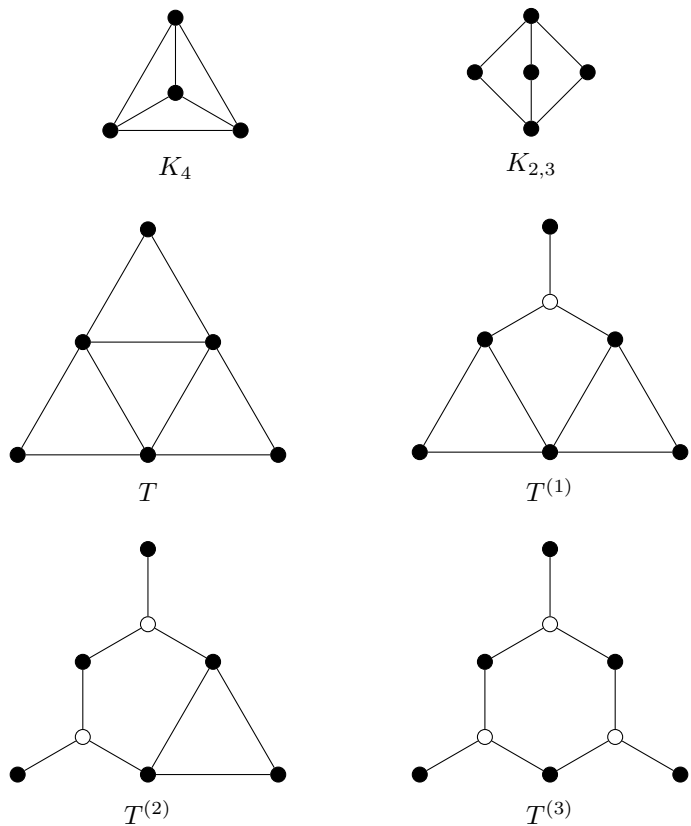
\begin{figure}[t]
\centering
\begin{tikzpicture}[every node/.append style={fill}]
\node (0) at (0,0) {};
\foreach \i in {1,2,3} {
    \pgfmathsetmacro{\ang}{90 + 120*(\i - 1)}
    \node (\i) at (\ang:1) {};
    \draw (0) -- (\i);
}
\draw (1) -- (2) -- (3) -- (1);
\node[draw=none,fill=none,rectangle] at ([yshift=-0.5cm]0,-0.5) {$K_4$};
\end{tikzpicture}
\hfil
\begin{tikzpicture}[every node/.append style={fill}]
\pgfmathsetmacro{\h}{0.5 * (1 + sin(30))}
\node (0) at (0,-\h) {};
\node (1) at (0,\h) {};
\node (2) at (-\h,0) {};
\node (3) at (0,0) {};
\node (4) at (\h,0) {};
\foreach \x in {0,1} {
    \foreach \y in {2,3,4} {
        \draw (\x) -- (\y);
    }
}
\node[draw=none,fill=none,rectangle] at ([yshift=-0.5cm]0,-\h) {$K_{2,3}$};
\end{tikzpicture}

\bigskip

\begin{tikzpicture}[every node/.append style={fill}]
\foreach \i in {1,2,3} {
    \pgfmathsetmacro{\ang}{90 + 120*(\i - 1)}
    \node (\i) at (\ang:2) {};
}
\foreach \i in {4,5,6} {
    \pgfmathsetmacro{\ang}{30 + 120*(\i - 4)}
    \node (\i) at (\ang:1) {};
}
\draw (1) -- (5) -- (2) -- (6) -- (3) -- (4) -- (1);
\draw (4) -- (5) -- (6) -- (4);
\node[draw=none,fill=none,rectangle] at ([yshift=-0.5cm]0,-1) {$T$};
\end{tikzpicture}
\hfil
\begin{tikzpicture}[every node/.append style={fill}]
\foreach \i in {1,2,3} {
    \pgfmathsetmacro{\ang}{90 + 120*(\i - 1)}
    \node (\i) at (\ang:2) {};
}
\foreach \i in {4,5,6} {
    \pgfmathsetmacro{\ang}{30 + 120*(\i - 4)}
    \node (\i) at (\ang:1) {};
}
\node[fill=none] (7) at (90:1) {};
\draw (7) -- (1);
\draw (7) -- (5) -- (2) -- (6) -- (3) -- (4) -- (7);
\draw (5) -- (6) -- (4);
\node[draw=none,fill=none,rectangle] at ([yshift=-0.5cm]0,-1) {$T^{(1)}$};
\end{tikzpicture}

\bigskip

\begin{tikzpicture}[every node/.append style={fill}]
\foreach \i in {1,2,3} {
    \pgfmathsetmacro{\ang}{90 + 120*(\i - 1)}
    \node (\i) at (\ang:2) {};
}
\foreach \i in {4,5,6} {
    \pgfmathsetmacro{\ang}{30 + 120*(\i - 4)}
    \node (\i) at (\ang:1) {};
}
\node[fill=none] (7) at (90:1) {};
\node[fill=none] (8) at (210:1) {};
\draw (7) -- (1);
\draw (8) -- (2);
\draw (7) -- (5) -- (8) -- (6) -- (3) -- (4) -- (7);
\draw (6) -- (4);
\node[draw=none,fill=none,rectangle] at ([yshift=-0.5cm]0,-1) {$T^{(2)}$};
\end{tikzpicture}
\hfil
\begin{tikzpicture}[every node/.append style={fill}]
\foreach \i in {1,2,3} {
    \pgfmathsetmacro{\ang}{90 + 120*(\i - 1)}
    \node (\i) at (\ang:2) {};
}
\foreach \i in {4,5,6} {
    \pgfmathsetmacro{\ang}{30 + 120*(\i - 4)}
    \node (\i) at (\ang:1) {};
}
\node[fill=none] (7) at (90:1) {};
\node[fill=none] (8) at (210:1) {};
\node[fill=none] (9) at (330:1) {};
\draw (7) -- (1);
\draw (8) -- (2);
\draw (9) -- (3);
\draw (7) -- (5) -- (8) -- (6) -- (9) -- (4) -- (7);
\node[draw=none,fill=none,rectangle] at ([yshift=-0.5cm]0,-1) {$T^{(3)}$};
\end{tikzpicture}
\caption{The $T_3$-family with all possible roots marked as filled vertices.}
\label{fig:t3}
\end{figure}

\begin{definition}
Let $\ttf$ be the family of six graphs shown in \cref{fig:t3}, known as the $T_3$-family.  Let $\ttfr$  be the family of rooted graphs $(H,j)$ such that $H\in\ttf$ and $j$ is a cut-vertex of $H$.
\end{definition}

\begin{proposition}
\label{prop:t3}
Let $(G,i)\in\ttfr$.  Then $\excise(G,i) \geq 4$.
\end{proposition}
\begin{proof}
It is enough to find matrices $A\in\mptn(G)$ with $i$-SNIP and the $i$-nullity pair $(2,2)$, or $(3,2)$ by \cref{prop:kk1}, for each $(G,i)$.  For $K_4$ and $K_{2,3}$, we may choose the all-ones matrix in $\mptn(K_4)$ and the adjacency matrix of $K_{2,3}$, respectively.  For each of $T$, $T^{(1)}$, $T^{(2)}$, and $T^{(3)}$, the edge set can be partitioned by some $K_3$ containing a vertex of degree $2$ and some $K_{1,3}$ centered at a cut-vertex; then we choose the matrix as the sum of the all-ones matrix in $\mptn(K_3)$ and the adjacency matrices of the $K_{1,3}$, padding with appropriate zeros.  Such a construction is referred to as the \emph{star-clique sum} introduced in \cite{SK12}.  To be more precise, we choose the following matrices such that $A_0\in\mptn(K_4)$, $A_1\in\mptn(K_{2,3})$, $B_0\in\mptn(T)$, and $B_k\in\mptn(T^{(k)})$ for $k = 1,2,3$.  Moreover, for $B_k$ with $k = 1,2,3$, the cut-vertices are indexed by $7$, $8$, and $9$, if they exist.  

\[
    A_0 = \begin{bmatrix}
        1 & 1 & 1 & 1 \\
        1 & 1 & 1 & 1 \\
        1 & 1 & 1 & 1 \\
        1 & 1 & 1 & 1 
    \end{bmatrix}, \qquad
    A_1 = \begin{bmatrix}
        0 & 0 & 1 & 1 & 1 \\
        0 & 0 & 1 & 1 & 1 \\
        1 & 1 & 0 & 0 & 0 \\
        1 & 1 & 0 & 0 & 0 \\
        1 & 1 & 0 & 0 & 0
    \end{bmatrix}, 
\]
\[
    B_0 = \begin{bmatrix}
        2 & 1 & 1 & 0 & 1 & 1 \\
        1 & 1 & 1 & 0 & 0 & 0 \\
        1 & 1 & 2 & 1 & 1 & 0 \\
        0 & 0 & 1 & 1 & 1 & 0 \\
        1 & 0 & 1 & 1 & 2 & 1 \\
        1 & 0 & 0 & 0 & 1 & 1 \\
    \end{bmatrix}, \qquad
    B_1 = \begin{bmatrix}
        1 & 0 & 0 & 0 & 1 & 1 & 1 \\
        0 & 0 & 0 & 0 & 0 & 0 & 1 \\
        0 & 0 & 1 & 1 & 1 & 0 & 1 \\
        0 & 0 & 1 & 1 & 1 & 0 & 0 \\
        1 & 0 & 1 & 1 & 2 & 1 & 0 \\
        1 & 0 & 0 & 0 & 1 & 1 & 0 \\
        1 & 1 & 1 & 0 & 0 & 0 & 0
    \end{bmatrix},
\]
\[
    B_2 = \begin{bmatrix}
        1 & 0 & 0 & 0 & 1 & 1 & 1 & 0 \\
        0 & 0 & 0 & 0 & 0 & 0 & 1 & 0 \\
        0 & 0 & 0 & 0 & 0 & 0 & 1 & 1 \\
        0 & 0 & 0 & 0 & 0 & 0 & 0 & 1 \\
        1 & 0 & 0 & 0 & 1 & 1 & 0 & 1 \\
        1 & 0 & 0 & 0 & 1 & 1 & 0 & 0 \\
        1 & 1 & 1 & 0 & 0 & 0 & 0 & 0 \\
        0 & 0 & 1 & 1 & 1 & 0 & 0 & 0
    \end{bmatrix}, \qquad
    B_3 = \begin{bmatrix}
        0 & 0 & 0 & 0 & 0 & 0 & 1 & 0 & 1 \\
        0 & 0 & 0 & 0 & 0 & 0 & 1 & 0 & 0 \\
        0 & 0 & 0 & 0 & 0 & 0 & 1 & 1 & 0 \\
        0 & 0 & 0 & 0 & 0 & 0 & 0 & 1 & 0 \\
        0 & 0 & 0 & 0 & 0 & 0 & 0 & 1 & 1 \\
        0 & 0 & 0 & 0 & 0 & 0 & 0 & 0 & 1 \\
        1 & 1 & 1 & 0 & 0 & 0 & 0 & 0 & 0 \\
        0 & 0 & 1 & 1 & 1 & 0 & 0 & 0 & 0 \\
        1 & 0 & 0 & 0 & 1 & 1 & 0 & 0 & 0
    \end{bmatrix}.
\]

By direct computation, one may check that each of these matrices achieves the $i$-nullity pair $(3,2)$ and has $i$-SNIP for $i\neq 7,8,9$.  Therefore, when $(G,i)\in\ttfr$, it allows the nullity pair $(2,2)$ with $i$-SNIP, implying $\excise(G,i) \geq 4$.   
\end{proof}

One may verify the following observation directly case by case.

\begin{observation}
\label{obs:addedge}
Let $G\in\{T_3^{(1)}, T_3^{(2)}, T_3^{(3)}\}$ as in \cref{fig:t3} with $i$ a cut-vertex of $G$.  Let $i'$ be the leaf neighbor of $i$ in $G$.  Then by adding an edge $e$ from $i'$ to any vertex other than $i$, the resulting graph $(G+e,i)$ contains a rooted minor in $\ttfr$.  For example, let $G = T_3^{(3)}$, $i$ a cut-vertex, and $i'$ its leaf neighbor.  By adding an edge $e$ between $i'$ and another leaf, $(G+e,i)$ contains $(H,j)\in\ttfr$ as a rooted minor, where $H = K_{2,3}$ and $j$ is a vertex of degree $3$.
\end{observation}

\begin{figure}[t]
\centering
\begin{tikzpicture}
\begin{scope}[preimage/.style={rectangle, minimum height=0.5cm, minimum width=1cm, fill=white}]
\node[preimage, minimum height=0.75cm, minimum width=3cm, label={right:$W_c$}] (Wc) at (0,0) {};
\path (Wc.south) --node[midway,above,rectangle,draw=none] (beta) {$\beta$} ++ (1.5,0);
\draw ([xshift=1.5cm]Wc.south) [rounded corners=0.5cm]-- ++(0,0.5) -- ++(-1.5,0) [sharp corners]-- (Wc.south);
\path (Wc.south) --node[midway,above,rectangle,draw=none] (alpha) {$\alpha$} ++ (-1.5,0);
\draw ([xshift=-1.5cm]Wc.south) [rounded corners=0.5cm]-- ++(0,0.5) -- ++(1.5,0) [sharp corners]-- (Wc.south);
\node[fill=white, label={below:$i$}] at (Wc.south) {};

\node[preimage, label={left:$W_a$}] (Wa) at ([yshift=-1cm]Wc.south west) {};
\node[preimage, label={right:$W_b$}] (Wb) at ([yshift=-1cm]Wc.south east) {};
\node[preimage, label={right:$W_d$}] (Wd) at ([yshift=1cm]Wc.north) {};

\begin{scope}[on background layer, preimage edge/.style={very thick, shorten <=-3pt, shorten >=-3pt}]
\draw[preimage edge] (Wd.south) -- (Wc.north);
\draw[preimage edge] (alpha.south) -- (Wa.north);
\draw[preimage edge] (beta.south) -- (Wb.north);
\draw[very thick, dashed] (Wa.south) -- ++(0,-0.5);
\draw[very thick, dashed] (Wb.south) -- ++(0,-0.5);
\end{scope}

\node[draw=none, rectangle] at (0,-3) {$G$};
\end{scope}

\begin{scope}[xshift=5cm]
\node[label={right:$c$}] (c) at (0,0) {};
\node[label={left:$a$}] (a) at (210:1) {};
\node[label={right:$b$}] (b) at (-30:1) {};
\node[label={right:$d$}] (d) at (90:1) {};
\draw (c) -- (a);
\draw (c) -- (b);
\draw (c) -- (d);
\draw[dashed] (a.south) -- ++(0,-0.5);
\draw[dashed] (b.south) -- ++(0,-0.5);

\node[draw=none, rectangle] at (0,-3) {$H$};
\end{scope}
\end{tikzpicture}
\caption{An illustration of the proof of \cref{lem:notcut}.}
\label{fig:notcut}
\end{figure}
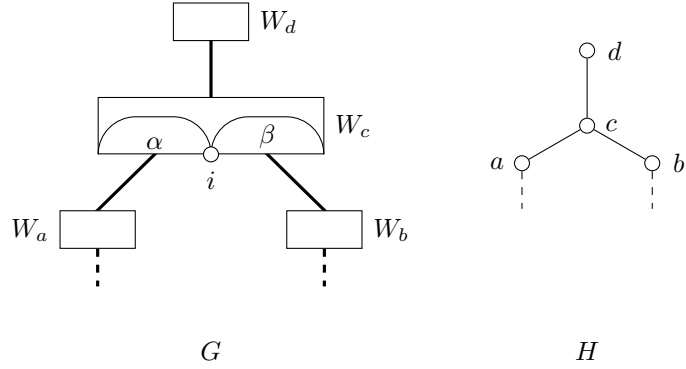

\begin{lemma}
\label{lem:notcut}
Let $G$ be a connected graph and let $i\in V(G)$ be a vertex that is not a cut-vertex of $G$.  Then the following are equivalent.
\begin{enumerate}[label={{\rm(\alph*)}}]
\item $(G,i)$ contains a rooted graph in $\ttfr$ as a rooted minor.
\item $G$ contains a graph $\ttf$ as a minor.
\end{enumerate}
\end{lemma}
\begin{proof}
If $(G,i)$ contains $(H,j)\in\ttfr$ as a rooted minor, then $G$ contains $H$ as a minor, while $H\in\ttf$ by definition.

Conversely, let $(G,i)$ be a rooted graph that does not contain any rooted graph in $\ttfr$ as a rooted minor.  Suppose, for the purpose of yielding a contradiction, that $G$ contains a graph $H\in\ttf$ as a minor.  Then there is a mapping $\varphi: W \rightarrow V(H)$ such that  
\begin{itemize}
\item $W\subseteq V(G)$, 
\item for each $v\in V(H)$, $\varphi^{-1}(\{v\}) \neq \emptyset$ and the induced subgraph $G[\varphi^{-1}(\{v\})]$ is connected, and 
\item for each edge $\{u,v\}\in E(H)$, there is at least an edge between $\varphi^{-1}(\{u\})$ and $\varphi^{-1}(\{v\})$ in $G$.
\end{itemize}
Since $G$ is connected, we may pick a mapping $\phi$ such that $W = V(G)$.  
For convenience, let $W_v = \varphi^{-1}(\{v\})$.  Note that $W_v$ are mutually disjoint and their union is $W = V(G)$.  Let $c$ be the vertex in $V(H)$ such that $i \in W_c$.  Then, by our assumption, $c$ must be a cut-vertex in $H$ and, necessarily, $H\in\{T^{(1)}, T^{(2)}, T^{(3)}\}$.  Let $a,b$ be the non-leaf neighbors of $c$ in $H$, and $d$ the leaf neighbor of $b$ in $H$.  See \cref{fig:notcut} for an illustration.

Let $\alpha$ be the set of all vertices in $W_c$ that are adjacent to some vertex in $W_a$; similarly, let $\beta$ be the set of all vertices in $W_c$ that are adjacent to some vertex in $W_b$.  Since $W_c$ induces a connected subgraph in $G$, there is a path in $W_c$ between any vertex in $\alpha$ to any vertex in $\beta$.  Let $P$ be such a path.  If $i$ is not on the path, then there is a path in $W_c$ from $i$ to $P$, and this gives a rooted minor $(H,d)$ in $(G,i)$.  Therefore, any path from $\alpha$ to $\beta$ must pass through $i$.

For any path $P$ in $W_c$ from $\alpha$ to $\beta$ and a vertex $w\in W_c\cup W_d\setminus V(P)$, there is a path $P'$ in $W_c\cup W_d$ that connects from $w$ to $P$ since the vertices $c$ and $d$ are adjacent in $H$.  If $P'$ first intersects with $P$ at a vertex $p$ other than $i$, then $(G,i)$ contains $(H,a)$ or $(H,b)$ as a rooted minor.  Therefore, for any such $P$ and $w\in W_c\cup W_d\setminus V(P)$, every path from $w$ to $P$ is through $i$.  

We claim that $i$ has to be a cut-vertex, violating the assumption.  Let $w\in W_d$ and $P'$ a path from $w$ to $\varphi^{-1}(V(H) - c - d)$.  Let $P$ be a path in $W_c$ from $\alpha$ to $\beta$.  If $P'$ does not intersect with $P$, then $(G,i)$ contains a rooted minor $(H',c)$, where $H'$ is obtained from $H$ by adding some edge between $d$ and $V(H) - c - d$, and $(H',c)$ contains a rooted minor in $\ttfr$ by \cref{obs:addedge}.  Therefore, $P'$ meets $P$ before reaching $\varphi^{-1}(V(H) - c - d)$.  Moreover, by the previous arguments, $i$ is the first vertex when $P'$ meets $P$.  Consequently, $i$ is a cut-vertex of $G$ that separates $W_d$ and $\varphi^{-1}(V(H) - c - d)$, which contradicts to our assumption.  

Therefore, our hypothesis that $G$ contains a minor in the $\ttf$ is impossible.  This completes the proof.
\end{proof}

Now we are ready for our main theorems.  

\begin{theorem}
\label{thm:e4}
Let $G$ be a connected graph and $i\in V(G)$ that is not a cut-vertex of $G$.  Then the following are equivalent.
\begin{enumerate}[label={{\rm(\alph*)}}]
\item $\excise(G,i) \geq 4$. \label{item:e41}
\item $(G,i)$ contains a rooted graph in $\ttfr$ as a rooted minor. \label{item:e42}
\item $G$ contains a graph $\ttf$ as a minor. \label{item:e43}
\item $\xi(G) \geq 3$. \label{item:e44} 
\end{enumerate}
\end{theorem}
\begin{proof}
By \cref{lem:notcut}, \ref{item:e42} and \ref{item:e43} are equivalent.  Also, it is known that \ref{item:e43} and \ref{item:e44} are equivalent \cite{HvdH07}.

By \cref{prop:t3}, \ref{item:e42} implies \ref{item:e41}.  By definition, $\excise(G,i) \ge 4$ means that $(G,i)$ allows the nullity pair $(2,2)$, and also $(3,2)$ by \cref{prop:kk1}, with $i$-SNIP, so $\xi(G) \geq 3$.  This establishes that \ref{item:e41} implies \ref{item:e44}, completing the proof that the four statements are equivalent.
\end{proof}

\begin{corollary}
\label{cor:e45}
The minimal minors for $\excise(G, i) \geq 4$ are $\ttfr$, while the minimal minors for $\excise(G, i) \geq 5$ are those obtained from $\ttfr$ by extending their roots. 
\end{corollary}
\begin{proof}
Let $(G,i)$ achieve $\excise(G, i) \geq 4$ as a minimal minor.
By \cref{prop:disjoint,thm:cutvtx}, $(G,i)$ is connected and $i$ is not a cut-vertex of $G$.  By \cref{thm:e4}, $(G,i)$ contains a rooted minor in $\ttfr$, which means $(G,i)\in\ttfr$ by the minimality of $(G,i)$.  With the minimal minors for $\excise(G, i) \geq  4$ established, the minimal minors for $\excise(G, i) \geq 5$ are immediate by \cref{thm:2k2kp1}. 
\end{proof}

\section{Many aspects of \texorpdfstring{$i$}{i}-SNIP}
\label{sec:snipthy}
In this section, we provide a comprehensive investigation of $i$-SNIP and its consequences.  We prove the Bifurcation Lemma introduced in \cref{sec:bif}.  We also introduce several basic tools for dealing with $i$-SNIP, including edge bounds and an equivalent alternative null space definition of it.

Let $(G,i)$ be a rooted graph and $A\in\mptn(G)$.  Let $\mptncl(G)$ be the \emph{topological closure} of $\mptn(G)$, i.e., all symmetric matrices whose $i,j$-entry is zero whenever $\{i,j\}\notin E(G)$ and $i\neq j$.
Note that the entries of a matrix in $\mptncl(G)$ that correspond to an edge are allowed to be zero.  Let $\GL{i}$ be the subgroup of the general linear group defined by 
\[
    \GL{i} := \{Q\in\mat : \det(Q) \neq 0,\ Q[i,i) = \bzero\trans\}.
\]
We consider two perturbations  
\[
    \begin{aligned}
    A &\mapsto A + B, \\
    A &\mapsto Q\trans AQ,
    \end{aligned}
\]
where $B\in\mptncl(G)$ and $Q\in\GL{i}$.  The first of these perturbations preserves the pattern of the matrix for a sufficiently small perturbation $B$, and the second of these preserves the $i$-nullity pair.  Both of them preserve the symmetry of the matrix.

In \cite{SNIP}, we considered the function  
\begin{equation}
\label{eq:supperturb}
    F(B,Q) = Q\trans AQ + B,
\end{equation}
defined for $B\in\mptncl(G)$ near $O$ and $Q\in\GL{i}$ near $I$.
For completeness, we include the conclusions that were obtained in \cite{SNIP} from the analysis of this function $F$.

\begin{proposition}[Proposition~6.4 of \cite{SNIP}]
\label{prop:snipspaces}
Let $G$ be a graph on $n$ vertices, $A\in\mptn(G)$, and $i\in V(G)$. Let $F$ be defined as in \cref{eq:supperturb},  and let $\dot{F}_B$, $\dot{F}_Q$ be the partial derivatives at $B = O$ and $Q = I$.  Then the following hold:
\begin{enumerate}[label={{\rm(\alph*)}}]
\item $\range(\dot{F}_B) = \mptncl(G)$.
\item $\range(\dot{F}_Q) = \{L\trans A + AL: L\in\mat,\ L[i,i) = \bzero\trans\}$.
\item $\range(\dot{F}_B)^\perp = \{X\in\msym: A\circ X = O,\ I\circ X = O\}$.
\item $\range(\dot{F}_Q)^\perp = \{X\in\msym: (AX)[i,i] = 0,\ (AX)(i,:] = O\}$.
\end{enumerate}
\end{proposition}

\begin{proposition}[Proposition~6.5 of \cite{SNIP}]
\label{prop:snipequiv}
Let $G$ be a graph on $n$ vertices, $A\in\mptn(G)$, and $i\in V(G)$.  Let $F$ be defined as in \cref{eq:supperturb} and let $\dot{F}$, $\dot{F}_B$, $\dot{F}_Q$ be the derivative and the partial derivatives at $B = O$ and $Q = I$.  Then the following are equivalent:
\begin{enumerate}[label={{\rm(\alph*)}}]
\item $A$ has $i$-SNIP.
\item $\dot{F}$ is surjective.
\item $\range(\dot{F}) = \range(\dot{F}_B) + \range(\dot{F}_Q) = \msym$.
\item $\range(\dot{F}_B)^\perp \cap \range(\dot{F}_Q)^\perp = \{O\}$.
\end{enumerate}
\end{proposition}

Both the previous results and the new results make use of the inverse function theorem as stated below.

\begin{theorem}[Inverse Function Theorem; see, e.g., Theorem~6.6 of \cite{SNIP}]
\label{thm:invftsur}
Let $U$ and $W$ be finite-dimensional vector spaces over $\mathbb{R}$. Let $F$ be a smooth function from an open subset of $U$ to $W$ with $F(\bu_0) = \bw_0$. If the derivative $\dot{F}$ at $\bu_0$ is surjective, then there is an open subset $W'\subseteq W$ containing $\bw_0$ and a smooth function $T:W'\rightarrow U$ such that $T(\bw_0) = \bu_0$ and $F\circ T$ is the identity map on $W'$. 
\end{theorem}

For the purposes of bifurcation, we redefine $F$ so that the two perturbations are applied in the opposite order,
giving us
\begin{equation}
\label{eq:bifperturb}
    F(B,Q) = Q\trans (A + B)Q
\end{equation}
for the remainder of the discussion.  Note that when $B = O$ or $Q = I$, the two definitions give the same function.
Therefore, they have the same partial derivatives with respect to $B$ and $Q$ at $B = O$ and $Q = I$.  Consequently, \cref{prop:snipspaces,prop:snipequiv} remain valid for \cref{eq:bifperturb}.

Next we restate the Bifurcation Lemma and present its proof.

\bifurcation*

\begin{proof}
Let $F$ be as defined in \cref{eq:bifperturb} and let $\dot{F}$ be its derivative at $B = O$ and $Q = I$.  Since $A$ has $i$-SNIP, $\dot{F}$ is surjective by \cref{prop:snipequiv}.  By \cref{thm:invftsur}, for each matrix $M\in\msym$ close enough to $A$, there are matrices $B'$ and $Q'$ such that $F(B',Q') = M$; that is,
\[
    F(B',Q') = Q^{\prime\top}(A + B')Q' = M.
\]
Thus, the matrix $A' = A + B'$ satisfies $A' = (Q^{\prime\top})^{-1}M(Q')^{-1}$. 
 When $B'$ and $Q'$ are small perturbations, $A'$ has the same pattern as $A$, i.e., $A'\in\mptn(G)$, and $A'$ has the same $i$-nullity pair as $M$.  

Since the perturbation is small, we still have   
\[
    \{L\trans A' + A'L: L\in\mat,\ L[i,i) = \bzero\trans\} + \mptncl(G) = \msym.
\]
Therefore, the matrix $A'$ has $i$-SNIP if we choose $M$ close enough to $A$.
\end{proof}

Next we turn our attention to derive an edge bound for $\excise(G,i)$.  Similar results appeared in \cite[Theorem~6.5]{HHMS10} and \cite[Corollary~29]{gSAP}. In particular, these results imply the bound $e(G) + 1 \geq \binom{\xi(G) + 1}{2}$, and below we will use this fact to obtain a similar bound for $\excise(G,i)$.
We denote the number of edges of $G$ by~$e(G)$.

\begin{theorem}
\label{thm:edgebound}
Let $(G,i)$ be a rooted graph $\excise(G,i) = k + \ell$ for some $k \leq \ell \leq k+1$.  Then 
\begin{enumerate}[label={{\rm(\alph*)}}]
\item $e(G) + 1 \geq \binom{k+2}{2}$ if $k = \ell$, and 
\item $e(G - i) + 1 \geq \binom{\ell + 1}{2}$ if $k < \ell$.
\end{enumerate}
Overall,
setting
$m = \Big\lceil \frac{\excise(G,i) + 3}{2} \Big\rceil$, the bound $e(G) + 1 \geq \binom{m}{2}$ holds.
\end{theorem}
\begin{proof}
If $k = \ell$, then $(G,i)$ allows $(k+1, \ell)$ with $i$-SNIP by \cref{prop:kk1}.  Therefore, $\xi(G) \geq k+1$ and $e(G) + 1 \geq \binom{k+2}{2}$ by \cite[Theorem~3.9]{SNIP}.

If on the other hand $k < \ell$, then by \cite[Theorem~3.9]{SNIP} we have $\xi(G - i) \geq \ell$ and $e(G - i) + 1 \geq \binom{\ell + 1}{2}$.

Finally, we observe the case $k = \ell$ gives
\[
    k + 2 = \frac{\excise(G,i) + 4}{2}
\]
and the case $k < \ell$ gives
\[
    \ell + 1 = \frac{\excise(G,i) + 3}{2}.
\]
Along with $e(G) \geq e(G - i)$, we obtain the general bound
\[
    e(G) + 1 \geq \binom{m}{2}
    \ \ \mbox{ where }\  m = \bigg\lceil \frac{\excise(G,i) + 3}{2} \bigg\rceil.
\]
\end{proof}

The previous result allows us to derive a Nordhaus--Gaddum type bound, an analogous result of \cite[Section 3]{BFHH2013}.

\begin{corollary}
\label{cor:nhbound}
Let $(G,i)$ be a rooted graph on $n$ vertices. Then
\[
    \xi\xi(G,i) + \xi\xi(\overline{G},i)\leq 2\sqrt{2}n.
\]
\end{corollary}
\begin{proof}
By \cref{thm:edgebound}, with $m$ as there defined, we have 
\[
    e(G) + 1 \geq
    \binom{m}{2}
    \ge 
    \ \frac{1}{8}(\excise(G,i)^2 + 4\,\excise(G,i) + 3).
\]
Let $x = \excise(G,i)$ , $\overline{x} = \excise(\overline{G},i)$, and $s = x + \overline{x}$.  Then we may take the sum of 
\[
    \begin{aligned}
        e(G) + 1 &\geq \frac{1}{8}(x^2 + 4x + 3) \text{ and}\\
        e(\overline{G}) + 1 &\geq \frac{1}{8}(\overline{x}^2 + 4\overline{x} + 3)
    \end{aligned}
\]
to get 
\[
    \binom{n}{2} + 2 \geq \frac{1}{8}(x^2 + \overline{x}^2 + 4s + 6).
\]
By the Cauchy--Schwarz inequality, 
\[
    (x^2 + \overline{x}^2)(1^2 + 1^2) \geq (x + \overline{x})^2 = s^2.
\]
Therefore, 
\[
    \binom{n}{2} + 2 \geq \frac{1}{8}(\frac{1}{2}s^2 + 4s + 6).
\]
By solving the inequality, we have 
\[
    s \leq -4 + \sqrt{16 + 8n(n-1)},
\]
which is bounded above by $\sqrt{8n^2} = 2\sqrt{2}n$ if $n \geq 0$.  
\end{proof}

When the Colin de Verdière parameter $\mu(G)$ was defined in \cite{CdV,CdVF}, SAP was introduced by an equivalent definition through the surjectivity of a certain map.
This was then translated to a null space characterization in \cite{ADH2021,H2010},
using the span of certain matrices, here called ``ingredients'', coming from either the vertices or the edges of the graph~$G$.
Since the characterization depends only on the subspace and the graph,
by extension it also allows for a variant of SAP to be defined on a subspace other than the kernel.

\newcommand{\uprod}[2]{\bu_{#1}^{\phantom{\trans}}\mspace{-3.5mu}
    \bu_{#2}^{\mspace{-3mu}\top}}

\begin{definition}
\label{def:recipe}
Let $G$ be a graph on $n$ vertices labeled $\{1, \dots, n\}$, and let $L$ be a subspace of $\mathbb{R}^n$ of dimension $m$. The subspace $L$ gives a \emph{full recipe from~$G$} if, given a matrix $N$ whose $m$ columns form a basis for $L$
and whose $n$ rows we call $\bu_1\trans, \dots, \bu_n\trans$, the following set of $m \times m$ real symmetric \emph{ingredients} spans
the entire space $\msym[m]$:
\begin{itemize}
\item $\uprod{j}{j}$ for $j$ any vertex of $G$ (called \emph{vertex ingredients}) and
\item $\uprod{j}{k} + \uprod{k}{j}$ for $\{j,k\}$ any edge of $G$ (called \emph{edge ingredients}).
\end{itemize}
\end{definition}

We remark that while the definition of a full recipe appears to depend on the choice of a basis matrix for $L$, it is nonetheless well-defined for any $L$.
Suppose that $N_1$ and $N_2$ are two basis matrices, implying $N_1 = N_2Q$ for some $m\times m$ invertible matrix $Q$.  Let $M\in\msym[m]$ be given.  Observe that $M$ is a linear combination of matrices of the form $\uprod{j}{j}$ and $\uprod{j}{k} + \uprod{k}{j}$ (ingredients from $N_1$) if and only if $Q\trans\! MQ$ is a linear combination of matrices of the form $Q\trans \!\uprod{j}{j}Q$ and $Q\trans \!\uprod{j}{k} Q + Q\trans \!\uprod{k}{j} Q$ (ingredients from $N_2$).
Note that the map $M \mapsto Q\trans \! MQ$ is a bijection from $\msym[m]$ to itself.
Whether the definition holds or not is thus independent of the choice of basis for the subspace $L$.

\begin{theorem}[Lemma~2.1 of \cite{H2010}; Lemma~14 of \cite{ADH2021}]
\label{thm:hein}
Let $G$ be a graph and $A\in\mptn(G)$.  Then $A$ has SAP if and only the kernel of $A$ gives a full recipe from~$G$.
\end{theorem}

Here we prove an analogous result for $i$-SNIP.

\begin{theorem}
\label{thm:nullspace}
Let $(G,i)$ be a rooted graph and $A\in\mptn(G)$.  Then $A$ has $i$-SNIP if and only if  the right kernel of $A(i, :]$ gives a full recipe from~$G$.
\end{theorem}
\begin{proof}
Suppose $i$ is a downer index of $A$.  Then $A$ has $i$-SNIP if and only if $A$ has SAP, by \cref{thm:snipchar}.
In this case row $i$ of $A$ is in the span of the remaining rows, so $\ker(A) = \ker(A(i,:])$.  By \cref{thm:hein}, $A$ has $i$-SNIP if and only if $\ker(A(i,:])$ gives a full recipe from $G$.  

Suppose $i$ is a neutral index of $A$.  Then there is a unique $t$ such that $i$ is a downer index of $A_t = A + tE_{i,i}$ and $A$ has $i$-SNIP if and only if $A_t$ has SAP by \cref{thm:snipchar}.  Similarly, $\ker(A_t) = \ker(A_t(i,:]) = \ker(A(i,:])$, so $A$ has $i$-SNIP if and only if $\ker(A(i,:])$ gives a full recipe from $G$.

Suppose $i$ is an upper index of $A$.  Then $A$ has $i$-SNIP if and only if $A(i)$ has SAP by \cref{thm:snipchar}.  In this case, column $i$ of $A(i,:]$ is not in the span of the remaining columns, so $\ker(A(i,:])$ is exactly the subspace obtained from $\ker(A(i))$ by padding a $0$ in entry~$i$.  That is, row $i$ of $N$ is zero, and $\ker(A(i))$ is fully spanned by $N(i,:]$.  Since row $i$ of $N$ is zero, for $\ker(A(i,:])$ the vertex ingredient coming from $i$ or any edge ingredient coming from $\{i,j\}$ is, in every case, the zero matrix.
By \cref{thm:hein}, $A$ has $i$-SNIP if and only if $\ker(A(i))$ gives a full recipe from $G - i$, if and only if $\ker(A(i,:])$ gives a full recipe from $G$.
\end{proof}

\section{Concluding remarks}
\label{sec:conc}
In this paper, we introduced a new parameter $\excise(G,i)$ as a tool to understand the nullity pair problem.  We established various theoretical results for $i$-SNIP, including the Bifurcation Lemma, an edge bound, and a null space definition of $i$-SNIP.
Using these results we were able to characterize the graphs with $\excise(G,i) \geq s$ for $s \in \{4,5\}$.
All known minimal minors for $\excise(G, i) \geq s$ are summarized in \cref{fig:minminor}.
\begin{question}
    \label{q:xi}
    In the known cases of minimal minors for $\excise(G, i) \ge 2k$, corresponding to values $k \in \{0, 1, 2\}$, the minimal minors are precisely the minimal minors for $\xi(G) \ge k + 1$, rooted at any vertex that is not a cut-vertex.
    Does this pattern continue for $k > 2$?
\end{question}
We proved that the minimal minors for $\excise(G, i) \geq 2k + 1$ and the minimal minors for $\excise(G, i) \geq 2k$ can be obtained from each other, reducing the cases to be studied by half.
(This implies in particular that any case of an answer to \cref{q:xi} would also resolve the case $\excise(G, i) \ge 2k + 1$.)

These results fit in the general framework of the inverse eigenvalue problem for a graph $G$ as extended to the $(\lambda, \mu)$-problem, which studies the simultaneous spectra both of a matrix described by a graph and of the principal submatrix corresponding to deletion of a single root vertex.


\subsection*{Acknowledgements}
This project started and was made possible by the 2021 IEPG-ZF Virtual Research Community at the American Institute of Mathematics (AIM), with support from the US National Science Foundation. 
Aida Abiad is supported by the Dutch Research Council (NWO) through the grant VI.Vidi.213.085. 
Mary Flagg and the AIM SQuaRE where this research was completed are partially supported by DMS grant no.~2331634 from the National Science Foundation. J. C.-H. Lin was partially supported by the National Science and Technology Council of Taiwan (grant no.~NSTC-112-2628-M-110-003 and grant no.~NSTC-113-2115-M-110-010-MY3). We thank our host AIM at Caltech for their hospitality.
We also thank Bryan Curtis for his involvement in the research that gave a foundation to the present results and for other useful discussions.



\end{document}